\documentclass[a4paper,twopage,reqno,12pt]{amsart}
\usepackage[top=30mm,right=30mm,bottom=30mm,left=30mm]{geometry}

\usepackage{lmodern}
\usepackage{graphicx}
\usepackage{amsmath}
\usepackage{amssymb}
\usepackage{amsfonts}
\usepackage{amsthm}
\usepackage{amstext}
\usepackage{amsbsy}
\usepackage{amsopn}
\usepackage{amscd}
\usepackage{enumerate}
\usepackage{color}
\usepackage[colorlinks]{hyperref}
\usepackage[hyperpageref]{backref}
\usepackage{multirow}
\usepackage{float}
\usepackage{url}
\usepackage{longtable}
\usepackage{lscape}
\usepackage[numbers,sort]{natbib}

\newtheorem{theorem}{Theorem}[section]
\newtheorem{corollary}[theorem]{Corollary}
\newtheorem{lemma}[theorem]{Lemma}
\newtheorem{proposition}[theorem]{Proposition}

\theoremstyle{definition}

\newtheorem{conjecture}{Conjecture}[section]

\theoremstyle{remark}

\numberwithin{equation}{section}

\newcommand{\Irr}{\textsf{Irr}}
\newcommand{\Aut}{\textsf{Aut}}
\newcommand{\Out}{\mathsf{Out}}
\newcommand{\In}{\mathsf{I}}
\newcommand{\St}{\mathsf{St}}
\newcommand{\Msf}{\textsf{M}}
\newcommand{\GAP}{\textsf{GAP} }
\newcommand{\ATLAS}{\textsf{ATLAS} }
\newcommand{\cd}{\mathsf{cd}}

\newcommand{\Z}{\mathbf{Z}}

\newcommand{\C}{\mathbf{C}}

\newcommand{\Cyc}{\mathrm{C}}
\newcommand{\D}{\mathrm{D}}

\newcommand{\PSL}{\mathrm{PSL}}
\newcommand{\PSp}{\mathrm{PSp}}
\newcommand{\PSU}{\mathrm{PSU}}
\newcommand{\POm}{\mathrm{P\Omega}}
\newcommand{\B}{\mathrm{B}}
\newcommand{\E}{\mathrm{E}}
\newcommand{\A}{\mathrm{A}}
\renewcommand{\S}{\mathrm{S}}
\newcommand{\M}{\mathrm{M}}
\newcommand{\Qr}{\mathrm{Q}}
\newcommand{\F}{\mathrm{F}}
\newcommand{\G}{\mathrm{G}}
\newcommand{\J}{\mathrm{J}}

\newcommand{\Gp}{\overline{G'}}
\newcommand{\U}{\overline{U}}
\newcommand{\I}{\overline{I}}
\newcommand{\Q}{\overline{Q}}

\newcommand{\e}{\epsilon}
\newcommand{\ve}{\varepsilon}
\newcommand{\vp}{\varphi}

\renewcommand{\leq}{\leqslant}
\renewcommand{\geq}{\geqslant}
\renewcommand{\bar}{\overline}

\providecommand{\mod}[1]{ \ (\mathrm{mod}\ #1)}

\begin{document}

\title[Groups with character degrees of almost simple Ree groups]{On groups with the same character degrees as almost simple groups with socle small Ree groups}

\author[S.H. Alavi]{Seyed Hassan Alavi}
\address{%
Seyed Hassan Alavi, Department of Mathematics,
 Faculty of Science,
 Bu-Ali Sina University, Hamedan, Iran}
 \email{alavi.s.hassan@basu.ac.ir}
 \email{alavi.s.hassan@gmail.com}


\subjclass[2010]{Primary 20C15; Secondary 20D05}
\keywords{Huppert's Conjecture, Almost simple groups, Character degrees.}

\begin{abstract}
  Let $G$ be a finite group and $\cd(G)$ denote the set of complex irreducible character degrees of $G$. In this paper, we prove that if $G$ is a finite group and $H$ is an almost simple group with socle $H_{0}= \, ^{2}\G_{2}(q)$, where  $q=3^{f}$ with $f\geq 3$ odd such that $\cd(G)=\cd(H)$, then $G$ is non-solvable and the chief factor $G'/M$ of $G$ is isomorphic to $H_{0}$. If, in particular, $f$ is coprime to $3$, then $G'$ is isomorphic to $H_{0}$ and $G/\Z(G)$ is isomorphic to $H$.
\end{abstract}

\date{\today}
\maketitle


\section{Introduction}\label{sec:Intro}

Let $G$ be a finite group, and let $\Irr(G)$ be the set of complex irreducible characters of $G$. Denote the set of character degrees of $G$ by $\cd(G)$, and when the context allows us the set of irreducible character degrees will be referred to as the set of character degrees. There are various examples showing that the set of character degrees of $G$ cannot completely determine the structure of $G$ even for nilpotent and solvable groups.
For example, the non-isomorphic groups $\D_8$ and $\Qr_8$ not only have the same set of character degrees, but also share the same character table and  $\cd(\Qr_{8})=\cd(\S_{3}) = \{1, 2\}$. However, in the late 1990s, Huppert \cite{a:Hupp-I} posed a conjecture which, if true, would sharpen the connection between the character degree set of a non-abelian simple group and the structure of the group.

\begin{conjecture}[Huppert] Let $G$ be a finite group, and let $H$ be a finite non-abelian simple group such that the sets of character degrees of $G$ and $H$ are the same. Then $G \cong H \times A$, where $A$ is an abelian group.
\end{conjecture}

The conjecture asserts that the non-abelian simple groups are essentially characterized by the set of their character degrees. This conjecture is verified for alternating groups, sporadic simple groups, some finite simple groups of Lie type including the  Ree groups $^{2}\G_{2}(q)$, see for example \cite{a:ADTW-2011,a:ADTW-2013,a:Hup-Alt,a:Hupp-I,a:Hupp-III-VIII,a:Hup-PSL4,a:HT-Moster,a:Wakefield-2011-2G2}. Since 2016, we have started to study the structure of groups with the same character degrees as almost simple groups including those with socle  $\PSL_{2}(q)$ \cite{a:D-PSL2-ch2,a:AD-2022-PSL2}, Suzuki groups \cite{a:ADM-2021-Sz}, Ree groups \cite{a:AD-2021-2F4} and sporadic simple groups \cite{a:ADJ-Spor,a:ADJ-Mat}. There are several examples provided in \cite{a:ADJ-Spor,a:ADJ-Mat,a:D-PSL2-ch2} which show that Huppert's conjecture cannot be extended to almost simple groups. If $H$ is an almost simple group with socle $H_{0}=\PSL_{2}(2^{f})$ with $f$ prime, Daneshkhah \cite{a:D-PSL2-ch2} proved that there is an abelian group $A$ such that $G/A$ is isomorphic to $H$. In particular, she proved that $G'=H_{0}$ and $A=\Z(G)$ with only one exception, namely, when $H=\Aut(\PSL_{2}(4))$ in which $G'=H_{0}$ or $2\cdot H_{0}$, see \cite[Theorem~1.1]{a:D-PSL2-ch2}. It seems that if $G$ is a group  with the same character degrees as an almost simple group $H$ with socle $H_{0}$, then $G'=H_{0}$ and $G/\Z(G)$ is isomorphic to $H$ with a few exceptions for which $G'$ is an extension of an abelian normal subgroup of $G$ by $H_{0}$ in which cases we still have $G/A\cong H$ with $A$ central in $G$. 
In this paper, we continue our investigation to support this statement which is the almost simple group version of Huppert's conjecture, and study groups whose character degrees are the same as those of almost simple groups with socle Ree groups:

\begin{theorem}\label{thm:main}
	Suppose that  $H$ is an almost simple group
	with socle $H_{0}=\,^{2}\G_{2}(q)$, where $q=3^{f}$ with $f\geq 3$ odd. Let $G$ be a finite group  with  $\cd(G)=\cd(H)$. Then  $G$ is nonsolvable and the chief factor $G'/M$ of $G$ is isomorphic to $H_{0}$. If $f$ is coprime to $3$, then $G'\cong H_{0}$ and $G/\Z(G)$ is isomorphic to $H$.
\end{theorem}

Note that Theorem~\ref{thm:main} for $H=\,^{2}\G_{2}(q)$ is proved  by Wakefield \cite{a:Wakefield-2011-2G2}, and so we need only to focus on the case where $^{2}\G_{2}(q)\neq H\leq \Aut(^{2}\G_{2}(q))$. In order to prove Theorem~\ref{thm:main}, it is worth noting that we do not need to determine all the character degrees of $H$. We use several useful facts proved in Lemma \ref{lem:deg} on the character degrees of $G$. For example, one of the key result is the fact that the  $2$-part of the character degrees of $G$ divides $2^{3}$. 
It is unfortunate to state that our method does not work when $3\mid f$, in which case, we think that $G'$ may not be isomorphic to $H_{0}$ but  we still believe that $G/A$ is isomorphic to $H$,  for some abelian subgroup $A$ of $G$.  


\section{Preliminaries}\label{sec:prem}

In this section, we present some useful results to prove Theorem~\ref{thm:main}. We first establish some definitions and notation. Throughout this paper all groups are finite. A finite group of order $n$ is denoted by ``$n$''.
We use $\Cyc_{n}$ to denote the  cyclic group of order $n$, and $\A_{n}$ and $\S_{n}$ are the alternating group and the symmetric group on $n$ letters, respectively. We follow standard notation as in \cite{b:Atlas} for finite simple groups.
A  group $H$ is said to be an almost simple group
with socle $H_{0}$ if $H_{0}\leq H\leq \Aut(H_{0})$,
where $H_{0}$ is a non-abelian simple group.
If $N\unlhd G$ and $\theta\in \Irr(N)$, then the \emph{inertia group} (or the  stabilizer of $\theta$ in $G$) $\In_G(\theta)$ of
$\theta$ in $G$ is defined by $\In_G(\theta)=\{g\in G \mid \theta^g=\theta\}$. If the character $\chi=\sum_{i=1}^k e_i\chi_i$, where each $\chi_i$ is an irreducible character of $G$ and $e_i$ is a non-negative integer, then those $\chi_i$ with $e_i>0$ are called the \emph{irreducible constituents} of $\chi$.
All further notation and definitions are standard and can be found in \cite{b:Hupp-Char,b:Isaacs}. For the computation parts, we use GAP~\cite{GAP}.

\begin{lemma}[{\cite[Theorems 19.5 and 21.3]{b:Hupp-Char}}]\label{lem:gal}
	Suppose that $N\unlhd G$ and $\chi\in \Irr(G)$.
	\begin{enumerate}[\rm (i)]
		\item[\rm (i)] If $\chi_N=\theta_1+\cdots+\theta_k$ with $\theta_i\in {\rm{\Irr}}(N)$, then $k$ divides $| G/N|$. In particular, if $\chi(1)$ is prime to $|G/N|$, then $\chi_N\in \Irr(N)$.
		\item[\rm (ii)] (Gallagher's Theorem) If $\chi_N\in \Irr(N)$, then $\chi\psi\in \Irr(G)$ for all $\psi\in \Irr(G/N)$.
	\end{enumerate}
\end{lemma}

\begin{lemma}[{\cite[Theorems 19.6 and 21.2]{b:Hupp-Char}}]\label{lem:clif}
	Suppose that $N\unlhd G$ and $\theta\in {\rm{\Irr}}(N)$. Let $I=\In_G(\theta)$.
	\begin{enumerate}
		\item[\rm (i)]  If $\theta^I=\sum_{i=1}^k\phi_i$ with $\phi_i\in {\rm{\Irr}}(I)$, then $\phi_i^G\in {\rm{\Irr}}(G)$. In particular, $\phi_i(1)|G:I|\in \cd(G)$.
		\item[\rm (ii)] If $\theta$ extends to $\varphi\in {\rm{\Irr}}(I)$, then $(\varphi\tau )^G\in {\rm{\Irr}}(G)$ for all $\tau\in {\rm{\Irr}}(I/N)$. In particular, $\theta(1)\tau(1)|G:I|\in {\rm{\cd}}(G)$.
	\end{enumerate}
\end{lemma}

A character $\chi\in \textrm{\Irr}(G)$ is said to be \emph{isolated} in $G$ if $\chi(1)$ is divisible by no proper non-trivial character degree of $G$ and no proper multiple of $\chi(1)$ is a character degree of $G$. In this situation, we also say that $\chi(1)$ is an \emph{isolated degree} of $G$.
We define a \emph{proper power} degree of $G$ to be a character degree of $G$ of the form $c^{a}$ for integers $c$ with $a>1$. A \emph{mixed} degree of $G$ is a character degree which is divisible by at least two distinct primes.  

\begin{lemma}[{\cite[Lemma~3]{a:HT-Moster}}]\label{lem:factsolv} Let $G/N$ be a solvable factor group of
	$G$ minimal with respect to being non-abelian. Then two cases can occur.
	\begin{enumerate}
		\item[\rm (i)] $G/N$ is an $r$-group for some prime $r$. In this case, $G$ has a proper prime power degree.
		Hence there exists $\psi\in {\rm{\Irr}}(G/N)$ such that $\psi(1)=r^b>1$. If $\chi\in {\rm{\Irr}}(G)$ and $r\nmid \chi(1)$, then $\chi\tau\in {\rm{\Irr}}(G)$ for all $\tau\in {\rm{\Irr}}(G/N)$;
		\item[\rm (ii)] $G/N$ is a Frobenius group with an elementary abelian Frobenius kernel $F/N$. Then $e:=|G:F|\in {\rm{\cd}}(G)$ and $|F/N|=r^a$ for some prime $r$ and $a$ is the smallest integer such that $r^a\equiv 1 \mod{e}$. If $\psi\in {\rm{\Irr}}(F)$, then either $f\psi(1)\in {\rm{\cd}}(G)$ or $r^a$ divides $\psi(1)^2$. In the latter case, $r$ divides $\psi(1)$.
		\begin{enumerate}
			\item[\rm (1)] If $\chi\in {\rm{\Irr}}(G)$ such that no proper multiple of $\chi(1)$ is in ${\rm{\cd}}(G)$, then either $e$ divides $\chi(1)$, or $r^a$ divides $\chi(1)^2$;
			\item[\rm (2)] If $\chi\in {\rm{\Irr}}(G)$ is isolated, then $e=\chi(1)$ or $r^a\mid \chi(1)^2$.
		\end{enumerate}
	\end{enumerate}
\end{lemma}

\begin{table}
	\centering 
	\small
	\caption{Character degrees of finite simple groups of Lie type.}\label{tbl:simple}%
	\begin{tabular}{@{}llll@{}}
		\hline
		$S$ &
		Conditions &
		Label&
		Degree \\
		\hline
		%
		$\PSL_{m}(q)$&
		$m\geq 2$, $q\geq 5$ if $m=2$, $(m,q)\neq (3,2), (4,2)$& $\St$ &
		$q^{\frac{m(m-1)}{2}}$ \\
		&
		$m\geq 4$&
		$(1,m)$&
		$\frac{q(q^{m-1}-1)}{q-1}$\\
		$\PSU_{m}(q)$&
		$m\geq 3$, $(m,q)\neq (3,2),(4,2)$&
		$\St$ &
		$q^{\frac{m(m-1)}{2}}$  \\
		&
		$m\geq 4$&
		$(1,m)$&
		$\frac{q(q^{m-1}-(-1)^{m-1})}{q+1}$ \\
		$\PSp_{2m}(q)$&
		$m\geq 2$, $(m,q)\neq (2,2)$ & $\St$ &
		$q^{m^{2}}$ \\
		&
		$m\geq 2$&
		$\left(
		\begin{matrix}
			0 & m \\
			1& \\	
		\end{matrix}
		\right)$ &
		$\frac{q(q^{m}-1)(q^{m-1}+1)}{q-1}$ \\
		$\POm_{2m+1}(q)$ &
		$m\geq 3$, $q$ odd & $\St$ &
		$q^{m^{2}}$ \\
		&
		$m\geq 3$, $q$ odd&
		$\left(
		\begin{matrix}
			0 & m \\
			1& \\	
		\end{matrix}
		\right)$ &
		$\frac{q(q^{m}-1)(q^{m-1}+1)}{q-1}$ \\
		$\POm^{\epsilon}_{2m}(q)$ &
		$m\geq 4$, $\epsilon=\pm$ & $\St$ &
		$q^{m(m-1)}$ \\
		$^{2}\B_{2}(q)$ &
		$q=2^{2a+1}\geq 8$ & $\St$ &
		$q^{2}$ \\
		$^{3}\D_{4}(q)$ &
		& $\St$ &
		$q^{12}$ \\
		$\E_{6}^{\epsilon}(q)$ &
		$\epsilon=\pm$ & $\St$ &
		$q^{36}$ \\
		%
		%
		$\E_{7}(q)$ &
		& $\St$ &
		$q^{63}$  \\
		&
		&
		$\Phi_{7,46}$ &
		$q^{46}\Phi_{7}\Phi_{12}\Phi_{14}$ \\
		$\E_{8}(q)$ &
		& $\St$ &
		$q^{120}$  \\
		$\F_{4}(q)$ &
		& $\St$ &
		$q^{24}$  \\
		$^{2}\F_{4}(q)$ &
		$q=2^{2a+1}\geq 8$ & $\St$ &
		$q^{12}$ \\
		$\G_{2}(q)$ &
		$q\geq 3$ & $\St$ &
		$q^{6}$ \\
		$^{2}\G_{2}(q)$&
		$q=3^{2a+1}\geq 27$ & $\St$ &
		$q^{3}$ \\
		\hline
		\multicolumn{4}{l}{Note: $\Phi_{i}$ is the $i$th cyclotomic polynomial in terms of $q$.}
	\end{tabular}	
\end{table}

\begin{lemma}[{\cite[Lemma 5]{a:Bia}}] \label{lem:exten-2}
	Let $N$ be a minimal normal subgroup of $G$ so that $N\cong S^k$, where $S$ is a non-abelian simple group. If
	$\theta\in \Irr(S)$ extends to $\Aut(S)$, then $\theta^k\in \Irr(N)$ extends to $G$.
\end{lemma}

\begin{lemma}[{\cite[Lemma 6]{a:Hupp-I}}]\label{lem:schur}
	Suppose that $M\unlhd G'=G''$ and for every $\lambda\in {\rm{\Irr}}(M)$ with $\lambda(1)=1$, $\lambda^g=\lambda$ for all $g\in G'$. Then $M'=[M,G']$ and $|M/M'|$ divides the order of the Schur multiplier of $G'/M$.
\end{lemma}

\begin{lemma}\label{lem:exten} If $S$ is a non-abelian simple group, then there exists a non-trivial irreducible character of $S$
	that extends to $\Aut(S)$. Moreover, the following holds:
	\begin{enumerate}[\rm (i)]
		\item[\rm (i)] if $S$ is an alternating group $\A_{n}$ of degree at least $8$, then $S$ has an irreducible character degree divisible by $16$;
		\item[\rm (ii)] if $S$ is a simple group of Lie type as in {\rm Table~\ref{tbl:simple}} and $\chi$ is an irreducible character of $S$ labeled as in the third column of {\rm Table~\ref{tbl:simple}}, then  $\chi$ extends to $\Aut(S)$;
		\item[\rm (iii)] if $S$ is the Tits group $^{2}\F_{4}(2)'$ or a sporadic simple group except for $\M_{22}$, then $S$ has an irreducible character degree divisible by $16$ unless $S$ is $\J_{1}$ or $\M_{22}$.
	\end{enumerate}
	Furthermore, all the characters in {\rm (i)} and {\rm (iii)} are extendible to $\Aut(S)$.
\end{lemma}
\begin{proof}
	The existence of an extendible character of $S$ to $\Aut(S)$ follows from~\cite[Lemma 5]{a:Bia}. \smallskip
	
	\noindent (i)   For a partition $(n-s-r,s+1,r-1)$ of $n\geq 8$, it follows from \cite{a:Manz-1988} that $\S_{n}$ has an irreducible character labeled by $\chi_{r,s}$, where $r\geq 1$, $s\geq 0$ and $r+2s+1\leq n$. The character $\chi_{r,s}$ restricts irreducibly to $\A_{n}$ except when $s=0$ and $n=2r+1$ or $s=1$ and $n=2r+2$, and its degree is
	\begin{align*}
		\chi_{r,s}(1)=\binom{n}{s}\binom{n-s-1}{r-1}\frac{n-2s-r}{r+s}.
	\end{align*}
	If $n=4t+1$ for some $t\geq 2$, then $16$ divides $8t(t-1)$, and so $\chi_{1,2}(1)=8t(t-1)(4t+1)/3$ is divisible by $16$. If $n=4t+3$ for some $t\geq 2$, then $16$ divides $\chi_{3,2}(1)=8t(t-1)(2t+1)(4t+3)(4t-1)/5$, and if $n=2t$ for some $t\geq 4$, then $16$ divides $\chi_{2,1}(1)=8t(t-1)(t-2)/3$. Therefore, for a given $n$, at least one of the  degrees $\chi_{1,2}(1)$, $\chi_{2,1}(1)$ and $\chi_{3,2}(1)$ is divisible by $16$. \smallskip
	
	\noindent (ii) For a finite simple group $S$ listed as in the first column of Table~\ref{tbl:simple}, the irreducible character labeled by $\St$ in Table~\ref{tbl:simple} is the  Steinberg character of $S$, and by \cite{a:Schmid85,a:Schmidt92},  $\St$ extends to $\Aut(S)$. The remaining characters with degrees in Table~\ref{tbl:simple} are the unipotent characters of $S$ which are extendible to $G$ by \cite[Theorems 2.4-2.5]{a:Malle-2008-ExtUnip}. \smallskip
	
	\noindent (iii) This part follows from \ATLAS \cite{b:Atlas} and \GAP \cite{GAP}.
\end{proof}

\begin{lemma}\label{lem:as}
	Let $G$ be a finite group with $G'$ being a non-abelian finite simple group. Then $G'\C_G(G')=G'\times \C_G(G')$, $\C_G(G')=\Z(G)$ and $G/\Z(G)$ is an almost simple group with socle $G'$.   	
\end{lemma}
\begin{proof}
	Let $A:= \C_G(G')$ and $T:=G'A$. Since $G'$ is a non-abelian simple group, $G'\cap A=1$, and hence $G'A=G'\times A$. As $G'\cap A=1$, it also follows that $[G,A]=1$, and hence $A=\Z(G)$. Moreover, $G'\cong G'\Z(G)/\Z(G)\unlhd G/\Z(G)\leq \Aut(G')$, that is to say, $G/\Z(G)$ is an almost simple group with socle $G'$.	
\end{proof}

\section{Proof of the main result}\label{sec:proof}

In this section, we prove Theorem \ref{thm:main}. We will denote the small Ree group by $^{2}\G_{2}(q)$ with $q=3^{2n+1}\geq 27$. These groups have been introduced by Ree in \cite{a:Ree-60,a:Ree-61}. We mainly follow the description of these groups and their subgroups given in \cite{a:Kemper-Ree,a:Levchuk-Ree,a:Ward-66} with a few exceptions in our notation. In what follows, we provide some information on small Ree groups and their character degrees.

\begin{lemma}[{\cite[Table 5.1.A and Theorem 5.1.4]{b:KL-90}}]\label{lem:out}
	Let $H_{0}=\, ^{2}\G_{2}(q)$ with $q=3^{f}\geq 27$ and $f$ odd. Then the outer automorphism group $\Out(H_{0})$ of $H_{0}$ is isomorphic to the  cyclic group $\Cyc_{f}$ generated by the field automorphism and the Schur multiplier $\Msf(H_{0})$ of $H_{0}$ is trivial.
\end{lemma}

\begin{lemma}\label{lem:deg}
	Let $H_{0}=\, ^{2}\G_{2}(q)$ with $q=3^{f}\geq 27$ and $f$ odd, and let $H_{0}\leq H\leq \Aut(H_{0})$ with $|H:H_{0}|=d$. Then
	\begin{enumerate}[\rm (i)]
		\item[\rm (i)] the irreducible character degrees of $H_{0}$ are recorded in {\rm Table~\ref{tbl:deg}}, and the irreducible characters with degrees as in one of the lines $2$, $3$, $4$ and $9$ extend to $H$ which are consequently character degrees of $H$;
		\item[\rm (ii)] if $\chi$ is an irreducible character of $H$ of prime power degree, then $\chi(1)=q^{3}$;
		\item[\rm (iii)] if $\chi$ is an irreducible character of $H$, then the $2$-part of $\chi(1)$ divides $2^{3}$;
		\item[\rm (iv)] the smallest even degree of $H$ is $\sqrt{3q}(q-1)(q+1)/3$;
		\item[\rm (v)] the group $H$ has an irreducible character of degree $(q^{3}+1)d$.
	\end{enumerate}
\end{lemma}

\begin{table}
	\centering 
	\small 
	\caption{The character degrees of the finite simple group  $^{2}\G_{2}(q)$ with $q=3^{f}\geq 27$ and $f$ odd.}\label{tbl:deg}%
	\begin{tabular}{@{}lll@{}}
		\hline
		Line &
		Degree &
		Comment \\
		\hline
		$1$ &$1$& \\
		$2$ &$(q-\sqrt{3q}+1)(q+\sqrt{3q}+1)$ & \\
		$3$ &$\frac{\sqrt{3q}}{6}(q-1)(q-\sqrt{3q}+1)$ & The character extends to  $\Aut(^{2}\G_{2}(q))$ \\
		$4$ &$\frac{\sqrt{3q}}{6}(q-1)(q+\sqrt{3q}+1)$ & The character extends to  $\Aut(^{2}\G_{2}(q))$\\
		$5$ &$\frac{\sqrt{3q}}{3}(q-1)(q+1)$ & The character extends to $\Aut(^{2}\G_{2}(q))$\\
		$6$ &$(q-1)(q+1)(q-\sqrt{3q}+1)$ & \\
		$7$ &$(q-1)(q-\sqrt{3q}+1)(q+\sqrt{3q}+1)$ & \\
		$8$ &$q(q-\sqrt{3q}+1)(q+\sqrt{3q}+1)$ & \\
		$9$ &$q^{3}$ & The character extends to $\Aut(^{2}\G_{2}(q))$\\
		$10$ &$(q+1)(q-\sqrt{3q}+1)(q+\sqrt{3q}+1)$ & \\
		$11$ &$(q-1)(q+1)(q+\sqrt{3q}+1)$ & \\					
		\hline
	\end{tabular}
\end{table}

\begin{proof}
	\noindent (i)
	The irreducible character degrees of $H_{0}$ can be read off from \cite{a:Ward-66}. The irreducible characters whose degrees are as in one of the lines $2$, $3$, $4$ and $9$ are the unipotent characters of $H_{0}$, see \cite[pp. 488-489]{b:Carter-85}. Therefore, these characters extend to $H$ by \cite[Theorems 2.4-2.5]{a:Malle-2008-ExtUnip}.\smallskip 
	
	\noindent (ii) We know by Lemma~\ref{lem:gal} that each degree of $H$ is a multiple of some character degree of $H_{0}$. We also observe that
	\begin{align}\label{eq:gcd}
		\nonumber	&(q-1,q+1)=2,\\
		&(q+\e_{1}1,q+\e_{2}\sqrt{3q}+1)=1,\\
		\nonumber	&(q-\sqrt{3q}+1,q+\sqrt{3q}+1)=1,
	\end{align}
	where $\e_{1}=\pm$ and $\e_{2}=\pm$. Therefore, the only prime power degree of $H_{0}$ is $q^{3}$ which is also extendible to $H$. This shows that the only prime power degree of $H$ is $q^{3}$.\smallskip
	
	\noindent (iii) It follows from \cite[Theorem]{a:Ward-66} that the Sylow $2$-subgroups of $H_{0}$ are elementary abelian groups of order $8$. Indeed, the $2$-part of $q-1=3^{f}-1$ is $2$ and the  $2$-part of $q+1=3^{f}+1$ is $4$. Since $f$ is odd and the irreducible character degree $\chi(1)$ of $H$ is a multiple of a divisor of $f$ and an irreducible character degree of $H_{0}$, it follows the $2$-part of $\chi(1)$ divides $8$. 
	
	\smallskip
	
	\noindent (iv) Note that the $2$-part of $q-1=3^{f}-1$ is $2$. Then the degrees in lines $3$ and $4$ of Table~\ref{tbl:deg} are odd. Moreover, the irreducible character of $H$ whose degree is a multiple of $(q-\sqrt{3q}+1)(q+\sqrt{3q}+1)$ and a divisor of $f$ is also odd. Therefore, $\sqrt{3q}(q-1)(q+1)/3$ is the smallest even irreducible character degree of $H_{0}$ which is also a degree of $H$ by part (i).\smallskip
	
	\noindent (v) Let $\iota$ be an involution in $H_{0}$. Then the centralizer of $\iota$ is (isomorphic to) $\langle \iota \rangle\times \PSL_{2}(q)$. We now consider an element $x$ of $\PSL_{2}(q)$ of order $(q-1)/2$ which is conjugate to its inverse but to no other power. By replacing $R$ with $x$ and $r$ with $\ve$ in \cite[Theorem]{a:Ward-66}, $H_{0}$ has the irreducible character $\eta_{\ve}$ of degree $q^{3}+1$, where $\ve$ is $(q-1)/2$th root of unity. We also observe by \cite[Ch. I]{a:Ward-66} that $\eta_{\ve}(x^{a})=\ve^{a}+\ve^{-a}$ which is obtained by the irreducible character of $\PSL_{2}(q)$ denoted $\theta_{\ve}$ in \cite[Ch. I]{a:Ward-66}. Moreover, distinct characters $\eta_{\ve}$ differ only on the classes of $x^{a}\neq 1$ and $\iota x^{a}\neq \iota$, see \cite[Ch. V]{a:Ward-66}. Recall from Lemma~\ref{lem:out} that the field automorphism $\varphi$ of $H_{0}$ induced by the Frobenius automorphism of $\mathbb{F}_{q}$ is of order $f$ and generates $\Out(H_{0})$, where $q=3^{f}\geq 27$ with $f$ odd.
	We now prove that $\eta_{\ve}$ is not invariant under any outer automorphism of $H_{0}$. Assume to the contrary that $\vp^{i}$ stabilizes $\eta_{\ve}$ for some $1\leq i<f$. Then $\eta_{\ve}^{\vp^{i}}(g)=\eta_{\ve}(g)$, for all $g\in H_{0}$. In particular, $\eta_{\ve}^{\vp^{i}}(x)=\eta_{\ve}(x)$. Then  $\ve^{3^{i}}+\ve^{-3^{i}}=\ve+\ve^{-1}$. By \cite[Lemma~4.7]{a:White2013}, this is equivalent to $3^{i}\equiv \pm1  \pmod{\frac{q-1}{2}}$. So $(q-1)/2=(3^{f}-1)/2$ divides $3^{i}\mp 1$, which is impossible. Therefore, the  elements of $H_{0}$ are the only elements of $H$ which stabilize $\eta_{\ve}$, in other words, $\In_{H}(\eta_{\ve})=H_{0}$. Therefore, $\eta_{\ve}^{H}\in \Irr(H)$, and hence $(q^{3}+1)d=\eta_{\ve}^{H}(1)\in \cd(H)$, as desired. 
\end{proof}

\begin{lemma}\label{lem:maxs}
	Let $H_{0}=\,^{2}\G_{2}(q)$ with $q=3^{2n+1}\geq 27$,  and let $H_{0}\leq H\leq \Aut(H_{0})$ with $|H:H_{0}|=d$. If $K$ is a maximal subgroup of $H_{0}$ whose index in $H_{0}$ divides some degree $\chi(1)$ of $H$, then $K=[q^{3}]: \Cyc_{q-1}$ and $\chi(1)=(q^{3}+1)a$ for some divisor $a$ of $d$.
\end{lemma}

\begin{table}
	\centering
	\small
	\caption{Maximal subgroups of the finite simple group  $^{2}\G_{2}(q)$ with $q=3^{f}\geq 27$ and $f$ odd.}\label{tbl:maxs}%
	\begin{tabular}{@{}lll@{}}
		\hline
		Maximal subgroup structure & 
		Order & 
		Index \\
		\hline
		$[q^{3}]:\Cyc_{q-1}$ &
		$q^{3}(q-1)$ &
		$q^{3}+1$\\
		$2\times \PSL_{2}(q)$ &
		$q(q^{2}-1)$&
		$q^{2}(q^{2}-q+1)$\\
		$2^{2}\times\D_{(q+1)/2}:\Cyc_{3}$ &
		$6(q+1)$&
		$\frac{1}{6}q^{3}(q-1)(q^{2}-q+1)$ \\
		$\Cyc_{q+\sqrt{3q}+1}:\Cyc_{6}$ & $6(q+\sqrt{3q}+1)$&
		$\frac{1}{6}q^{3}(q+1)(q-1)(q-\sqrt{3q}+1)$\\
		$\Cyc_{q-\sqrt{3q}+1}:\Cyc_{6}$& $6(q-\sqrt{3q}+1)$ &
		$\frac{1}{6}q^{3}(q+1)(q-1)(q+\sqrt{3q}+1)$\\
		$^{2}\G_{2}(q_{0})$, $q=q_{0}^{r}$, $r$ prime & $q_{0}^{3}(q_{0}^{3}+1)(q_{0}-1)$ &  $\frac{q_{0}^{3r}(q_{0}^{3r}+1)(q_{0}^{r}-1)}{q_{0}^{3}(q_{0}^{3}+1)(q_{0}-1)}$\\
		\hline
	\end{tabular}
\end{table}

\begin{proof}
	The list of maximal subgroups of $^{2}\G_{2}(q)$ is obtained by Kleidman \cite{a:Kleidman-88-G2-2G2-max}. We summarized the list of these subgroups in Table~\ref{tbl:maxs}. Assume first that $K$ is neither a parabolic subgroup $[q^{3}]:\Cyc_{q-1}$, nor a subfield subgroup $^{2}\G_{2}(q_{0})$. Then the index of the subgroup $K$ in $H_{0}$ must divide some mixed degrees of $H$. It is easy to observe that the $p$-part of the indices of these subgroups have exponents on $p$ too large to divide a mixed degree of $H$.
	If $K=\, ^{2}\G_{2}(q_{0})$ with $q=q_{0}^{r}$ and $r$ prime, then $|H_{0}:K|_{3}=q^{3r-3}$ divides $qd_{3}=q_{0}^{r}d_{3}$. Since $d_{3}<q=q_{0}^{r}$, we have that $3r-3< r$ implying that $r=1$, which is impossible. Therefore, $K$ is the parabolic subgroup $[q^{3}]:\Cyc_{q-1}$ whose index divides the degree $(q^{3}+1)a$ for some divisor $a$ of $d$.
\end{proof}

\begin{proposition}\label{prop:step-1}
	Let $G$ be a finite group, and let $H$ be an almost simple group whose socle is $H_{0}=\, ^{2}\G_{2}(q)$ with $q=3^{f}\geq 27$ and $f$ odd, and $|H:H_{0}|=d\neq 1$. If $\cd(G)=\cd(H)$, then $G'=G''$.
\end{proposition}
\begin{proof}
	Assume to the contrary that $G'\neq G''$. Then there is a normal subgroup $N$ of $G$, where $N$ is maximal such that $G/N$ is a non-abelian solvable group. We now apply Lemma~\ref{lem:factsolv} and have one of the  following two possibilities which leads to a contradiction.\smallskip
	
	\noindent (i) $G/N$ is a $r$-group with $r$ prime. In this case, $G/N$ has an irreducible character $\psi$ of degree $r^b>1$,  so does $G$. Lemma~\ref{lem:deg}(ii) implies that $\psi(1)=q^{3}$. We consider $\chi(1)=(q-1)(q+1)(q+\sqrt{3q}+1)a\in \cd(G)$ for some positive integer $a>1$. Then by Lemma~\ref{lem:gal}(i), we have that  $\chi_N\in \Irr(N)$, and so Lemma~\ref{lem:gal}(ii) implies that $G$ has an irreducible character of degree $q^{3}(q-1)(q+1)(q+\sqrt{3q}+1)a$, which is a contradiction.\smallskip
	
	\noindent (ii) $G/N$ is a Frobenius group with kernel $F/N$. Then $|G:F| \in \cd(G)$ divides $r^{a}-1$, where $|F/N|=r^{a}$. Note that $q^{3}$ is an isolated degree of $G$. Then Lemma~\ref{lem:factsolv}(ii.2) implies that $|G:F|=q^{3}$ or $r=3$. Let $\chi_{1}$ be an irreducible character of $G$ of degree $\chi_{1}(1)=\frac{\sqrt{3q}}{3}(q-1)(q+1)$.  Let also $c_{1}$ and $c_{2}$ be the  largest prime divisors of $d=|H:H_{0}|$ such that $\chi_{2}(1)=q(q-\sqrt{3q}+1)(q+\sqrt{3q}+1)c_{1}$ and $\chi_{3}(1)=(q-1)(q+1)(q+\sqrt{3q}+1)c_{2}$ being irreducible character degrees of $G$. Then no proper multiple of each $\chi_{i}(1)$ is in $\cd(G)$.
	
	Suppose that $|G:F|=q^{3}$. Then by Lemma~\ref{lem:factsolv}(ii.1), we conclude that $r^{a}$ divides the greatest common divisor of $\chi_{i}(1)^{2}$, for $i=1,2,3$, and so $r^{a}$ divides  $(c_{1},c_{2})^{2}$ by \eqref{eq:gcd}. Thus, $r^{a}$ divides $d^{2}$. On the other hand, $|G:F|=q^{3}$ divides $r^{a}-1$. Therefore, $q^{3}\leq r^{a}-1\leq d^{2}-1$, and hence, $3^{3f}\leq f^{2}-1$,  which is a contradiction.
	
	Suppose that $r=3$. Since $|G:F|\in \cd(G)$ and the greatest common divisor of $\chi_{1}(1)^{2}$, $\chi_{2}(1)^{2}$ and $\chi_{3}(1)^{2}$  divides $d^{2}$, it follows from Lemma~\ref{lem:factsolv}(ii.1) that $r^{a}$ divides $d^{2}$. Note that $d$ is a divisor of $|\Out(H_{0})|=f$. We again apply Lemma~\ref{lem:factsolv}(ii) and conclude that $|G:F|\leq r^{a}-1\leq f^{2}-1$, but then $|G:F|$ is too small to be a character degree of $G$, which is a contradiction.
\end{proof}

\begin{proposition}\label{prop:step-2}
	Let $G$ be a finite group, and let $H$ be an almost simple group whose socle is $H_{0}=\, ^{2}\G_{2}(q)$ with $q=3^{f}\geq 27$ and $f$ odd, and $|H:H_{0}|=d\neq 1$. Then the chief factor $G'/M$ of $G$ is isomorphic to $H_{0}$.
\end{proposition}
\begin{proof}
	Let $G'/M$ be a chief factor of $G$. As $G'$ is perfect, $G'/M$ is non-abelian, and so  $G'/M$ is isomorphic to $S^k$ for some non-abelian simple group $S$ and some integer $k\geq 1$. Note that each finite simple group has an even irreducible character degree. Then $2^{k}$ divides some degrees of $G$. By Lemma~\ref{lem:deg}(iii), we conclude that $k\leq 3$.
	
	Suppose that $S$ is an alternating group of degree $n\geq 8$, the Tits group $^{2}\F_4(2)'$ or a sporadic simple group except for the  Mathieu group $\M_{22}$ and the Janko group $\J_{1}$. Then by Lemma~\ref{lem:exten}(i) and (iii), $S$ has an irreducible character degree divisible by $16$, and so $G$ has a character degree divisible by $16$ but this violates Lemma~\ref{lem:deg}(iii). We now discuss the case where $S$ is $\A_{n}$ for $n=5,6,7$, $\M_{22}$ or $\J_{1}$.
	
	If $S=\A_{5}$, then $S$ has an irreducible character of degree $5$ which extends to $\S_{5}$. It follows from Lemma~\ref{lem:exten-2} that $G$ has an irreducible character of degree $5^{k}$, but this violates Lemma~\ref{lem:deg}(ii). If $S=\A_{6}$, the fact that $S$ has an irreducible character of degree $9$ which extends to $\A_{6}\cdot \Cyc_{2}^{2}$ implies that $G$ has an irreducible character of degree $3^{2k}$, and so by Lemma~\ref{lem:deg}(ii), we have that $3^{2k}=q^{3}=3^{3f}$ with $f$ odd,   which is a contradiction. If $S=\A_{7}$, then the irreducible character of $S$ of degree $6$ extends to $\Aut(S)$, and so by Lemma~\ref{lem:exten-2},  $G$ has an even degree $6^{k}$ for $k=1,2,3$. By Lemma~\ref{lem:deg}(iv), we conclude that $6^3\geq 6^{k}\geq \sqrt{3q}(q^{2}-1)/3$, which is impossible as $q=3^{f}\geq 27$ with $f$ odd.
	
	If $S=\J_{1}$, then $56,120\in \cd(S)$. This implies that $k=1$ as if not $G$ has an irreducible character degree divided by $56\cdot120=2^{6}\cdot 3\cdot 5\cdot 7$, which contradicts Lemma~\ref{lem:deg}(iii). Thus $k=1$, and hence $G$ has an even irreducible degree $56$. Now we apply Lemma~\ref{lem:deg}(iv), and so $56 \geq \sqrt{3q}(q^{2}-1)/3$, but this inequality  has no solution for $q=3^{f}\geq 27$. Similarly, if $S=\M_{22}$, then since $S$ has irreducible characters of even degrees $210$ and $280$, we conclude that $k=1$. The irreducible character of degree $210$ extends to $\Aut(S)=\M_{22}:2$, and so Lemma~\ref{lem:exten-2} implies that $210\in \cd(G)$ and this violates Lemma~\ref{lem:deg}(iv) as $210 < \sqrt{3q}(q^{2}-1)/3$ for all $q=3^{f}\geq 27$.
	
	Suppose now that $S=S(q_{0})$ is a finite simple group of Lie type except the Tits group $^{2}\F_{4}(2)'$, where $q_{0}=p^{e}$ for some positive integer $e$. In what follows, we frequently use Lemma~\ref{lem:exten}(ii) and the character degrees of $S$ recorded in Table~\ref{tbl:simple}.
	By Lemma~\ref{lem:exten}(ii), the Steinberg character of  degree $\St(1)=|S|_{p}$ extends to $\Aut(S)$, and so Lemma~\ref{lem:exten-2} implies that $G$ possesses a non-trivial prime power character degree $|S|_{p}^{k}$. Then Lemma~\ref{lem:deg}(v) forces
	\begin{align}\label{eq:st}
		|S|_{p}^{k}=3^{3f}.
	\end{align}
	This in particular implies that $p=3$. Since $f$ is odd and $k$ is $1$, $2$ or $3$, it also follows that $k$ is $1$ or $3$. We now discuss each possibilities of $S$ separately.
	
	Let $S$ be $\PSL_{m}(q_{0})$ with $m\geq 2$ or $\PSU_{m}(q_{0})$ with $m\geq 3$. Then by \eqref{eq:st}, we have that $|S|_{p}^{k}=\St(1)^{k}=3^{3f}$, and so $3^{em(m-1)k/2}=q_{0}^{m(m-1)k/2}=3^{3f}$. Thus $em(m-1)k=6f$ with $k=1,3$, and hence both $e$  and $m(m-1)/2$ are odd. Therefore, $m\geq 5$, or $S$ is $\PSL_{2}(q_{0})$, $\PSL_{3}(q_{0})$ or $\PSU_{3}(q_{0})$. If $m\geq 5$, then by Lemma~\ref{lem:exten}(ii), $S$ has a unipotent character $\chi$ of degree  $q_{0}(q_{0}^{m-1}-1)/(q_{0}-1)$ or $q_{0}(q_{0}^{m-1}-(-1)^{m-1})/(q_{0}+1)$ when $S$ is $\PSL_{m}(q_{0})$ or $\PSU_{m}(q_{0})$, respectively. The character $\chi$ extends to $\Aut(S)$, and so it follows from Lemma \ref{lem:exten-2} that $\chi(1)^{k}\in \cd(G)$.  Therefore, the $3$-part $q_{0}^{k}$ of $\chi(1)$ is equal to the $3$-part of some non-prime power irreducible degree of $G$. Then $q_{0}^{k}\geq \sqrt{3q}/3$, and so $3^{2ek+1}\geq 3^{f}$, and this implies that $2ek+1\geq f$. Since $em(m-1)k=6f$, it follows that $m(m-1)\leq 18$ which has no solution for  $m\geq 5$, which is a contradiction. Therefore, $m=2$ or $3$, that is to say, $S$ is $\PSL_{2}(q_{0})$, $\PSL_{3}(q_{0})$ or $S=\PSU_{3}(q_{0})$.
	Let $S=\PSL_{2}(q_{0})$. Then $q_{0}^{k}=q^{3}$ by \eqref{eq:st}. If $k=3$, then $q_{0}=q$, and so $S$ has irreducible character degrees $q$ and $q-1$. Therefore, $q^{2}(q-1)$ divides some degree of $G$, which is impossible. If $k=1$, then $q_{0}=q^{3}$, and so some degree of $G$ is divisible by $q_{0}-1=q^{3}-1$, which is a contradiction. Let $S=\PSL_{3}(q_{0})$ with $q_{0}=3^{e}\geq 3$ and $e$ odd.
	It follows from \cite{a:SL3} that $S$ has an irreducible character degree $(q_{0}-1)^{2}(q_{0}+1)$ which is divisible by $16$. Therefore, $16$ divides some degree of $G$, and this contradicts Lemma~\ref{lem:deg}(iii). Let $S=\PSU_{3}(q_{0})$ with $q_{0}=3^{e}\geq 3$ and $e$ odd. By
	\cite{a:SL3}, $S$ has an irreducible character degree $(q_{0}-1)(q_{0}+1)^{2}$ which is divisible by $16$ which is a contradiction as the $2$-part of any degree of $G$ divides $2^{3}$.
	
	Let $S$ be $\PSp_{2m}(q_{0})$ with $m\geq 2$ or $\POm_{2m+1}(q_{0})$ with $m\geq 3$. It follow from \eqref{eq:st}  that $q_{0}^{m^{2}k}=|S|_{p}^{k}=\St(1)^{k}=q^{3}$, and so $3^{em^{2}k}=3^{3f}$ implying that  $em^{2}k=3f$ with $k=1,3$. Thus $e$  and $m$ are odd. By Lemma~\ref{lem:exten}(ii), $S$ has a unipotent character $\chi$ of degree  $q_{0}(q_{0}^{m}-1)(q_{0}^{m-1}-1)/(q_{0}-1)$ which extends to $\Aut(S)$, and so  
	by applying Lemma \ref{lem:exten-2}, we conclude that $\chi(1)^{k}\in \cd(G)$.  This in particular implies that the $3$-part $q_{0}^{k}$ of $\chi(1)^{k}$ is equal to the $3$-part of some non-prime power irreducible degree of $G$. Then $q_{0}^{k}\geq \sqrt{3q}/3$, and so $2ek+1\geq f$. Recall that $em^{2}k=3f$. Then $6ek+1\geq em^{2}k$, and so $m^{2}<7$.  This is true only for $m=2$, which is a contradiction as $m$ is odd. 
	
	Let $S=\POm_{2m}^{\pm}(q_{0})$ with $m\geq 4$. Then \eqref{eq:st} implies that $3^{em(m-1)}=q_{0}^{m(m-1)}=3^{3f}$. Thus $em(m-1)=3f$, which is a contradiction as $f$ is odd but $m(m-1)$ is even. By the same argument, the case where $S$ is an exceptional simple group except for $\E_{7}(q_{0})$ and $^{2}\G_{2}(q_{0})$ can be ruled out, see Table~\ref{tbl:simple} for the character degrees. Let $S=\E_{7}(q_{0})$. Then by  \eqref{eq:st}, $q_{0}^{63k}=\St(1)^{k}=q^{3}$, and so $q_{0}=q^{\frac{3}{63k}}$. The  unipotent character $\Phi_{7,46}$ of $S$ extends to $\Aut(S)$ by  Lemma~\ref{lem:exten}(ii), and so   Lemma~\ref{lem:exten-2} yields $\Phi_{7,46}(1)^{k}\in \cd(G)$. Note that the $3$-part of $\Phi_{7,46}(1)^{k}$ is $q_{0}^{46k}=q^{3\cdot 46/63}$ which is greater than $q^{2}$. Hence $q^{2}$ divides the $3$-part of some non-prime power irreducible degree of $G$, which is impossible.
	
	Therefore, $S=\,^{2}\G_{2}(q_{0})$ with $q_{0}=3^{e}\geq 27$ and $e$ odd. Then $q_{0}^{3k}=|S|_{3}^{k}=q^{3}$, and so $q_{0}=q^{1/k}$. If $k=3$, then $q_{0}^{3}\cdot q_{0}^{2}(q_{0}^{2}-q_{0}+1)^{2}$ divides some degree of $G$, and so the $3$-part of this number which is $q_{0}^{5}=q^{5/3}$ has to divide the $3$-part of some non-prime power degree of $G$. Thus $q^{5/3}\leq qf$, and hence $q^{2}<f^{3}$, or equivalently, $3^{2f}<f^{3}$, but this inequality has no solution for odd positive integers $f\geq 3$. Therefore, $k=1$, and hence $q_{0}=q$, that is to say, $G'/M$ is isomorphic to $^{2}\G_{2}(q)$, as claimed.
\end{proof}

\begin{corollary}\label{cor:main}
	Let $G$ be a finite group with $\cd(G)=\cd(H)$, where $H$ is an almost simple group whose socle is $H_{0}=\, ^{2}\G_{2}(q)$ with $q=3^{f}\geq 27$ and $f$ odd. Then there exists a normal subgroup $N$ of $G$ such that $G/N$ is an almost simple group with socle $H_{0}$.
\end{corollary}
\begin{proof}
	By Proposition~\ref{prop:step-2}, the chief factor $G'/M$ is isomorphic to $H_{0}$. Let $N$ be a normal subgroup of $G$ such that $\Z(G/M)=N/M$. Applying Lemma~\ref{lem:as} to $G/M$, we conclude that $G/N$ is an almost simple group with socle $H_{0}$. 
\end{proof}

\begin{proposition}\label{prop:step-3}
	Let $G$ be a finite group with $\cd(G)=\cd(H)$, where $H$ is an almost simple group whose socle is $H_{0}=\, ^{2}\G_{2}(q)$ with $q=3^{f}\geq 27$ and $f$ odd, and $|H:H_{0}|=d\neq 1$. Let also the chief factor $G'/M$ be isomorphic to $H_{0}$. If $f$ is coprime to $3$, then  $M=1$, and hence $G'\cong H_{0}$.
\end{proposition}
\begin{proof}
	We first show that every linear character $\theta$ of $M$ is $G'$-invariant, that is to say, $\In_{G'}(\theta)=G'$. Assume to the contrary that $I:=\In_{G'}(\theta)<G'$ for some $\theta \in \Irr(M)$ with $\theta(1)=1$. Let $\theta^I= \sum_{i=1}^{k} \phi_i$, where $\phi_i \in \Irr(I)$ for $ i=1,2,...,k$. Suppose that $\U:=U/M$ is a maximal subgroup of $\Gp:=G'/M\cong H_{0}$ containing $\I:=I/M$, and set $t:= |U:I|$. It follows from Lemma \ref{lem:clif}(i) that $\phi_i(1)|G':I| \in \cd(G')$, and so $t\phi_i(1)|G':U|$ divides some character degree of $G$. Then $|G':U|$ must divide some character degree of $G$, and so by Lemma~\ref{lem:maxs}, $|G':U|$ divides $(q^{3}+1)a\in \cd(G)$ for some divisor $a$ of $d=|H:H_{0}|$ and the group $\U$ is $\Q: \Cyc_{q-1}$, where $\Q:=[q^{3}]$. Thus, $t\phi_i(1)$ divides $d$ which is a divisor of $f$. Therefore, $t=|U:I|=|\U:\I|$ divides $f$ which is coprime to $3$, and hence $\Q$ is contained in $\I$. Let $s:=|\I:\Q|$. Then $q-1=|\U:\Q|=ts$. Since $f$ is odd and $t$ is a divisor of $f$, we observe that $s$ is even, and hence $\I=\Q:\Cyc_{s}$, where $s=(q-1)/t$ is even.
	
	If $\varphi_{i}(1)=1$ for some $i$, then $\theta$ would extend to $\varphi_{i}$, and so Lemma~\ref{lem:clif}(ii) implies that $(\varphi_{i}\tau)^{G'}$ is an irreducible character of $G'$, for all $\tau\in \Irr(\I)$. We know that $|G':I|=(q^{3}+1)t$ and  $\I$ has a Frobenius subgroup $\bar{J}:=\Q:\Cyc_{s/2}$ which has an irreducible character of degree $s/2$, see \cite[Theorem(3)]{a:Ward-66}. Thus $(q^{3}+1)ts$ or $(q^{3}+1)ts/2$ is an irreducible character degree of $G'$, and hence some degree of $G$ is divisible by $(q^{3}+1)(q-1)/2$, which is a contradiction. Therefore, $\varphi_{i}(1)>1$ is an odd divisor of $a$, for all $i$. Since $f$ is coprime to $3$, each $\varphi_{i}(1)$ is  coprime to $3$.
	We claim that $Q/\ker(\theta)$ is abelian. To show this, it suffices to prove that $Q'\leq \ker(\varphi_{i})$, for all $i$, as if this true, then $Q'\leq \cap_{i=1}^{k}\ker(\varphi_{i})=\ker(\theta^{I})=\ker(\theta)$ which implies that $Q/\ker(\theta)$ is abelian.  Suppose that $Q'$ does not contain $\ker(\varphi_{j})$ for some $j$. Then $(\varphi_{j})_{Q}=\sum_{i} \sigma_{i}$ with $\sigma_{i}$ nonlinear in $\Irr(Q)$. Since $M/\ker(\theta)$ is central in $I/\ker(\theta)$, by Ito's theorem \cite[Theorem 19.9]{b:Hupp-Char}, we conclude that the non-trivial degrees of $Q/\ker(\theta)$ are powers of $3$, and hence  $\varphi_{j}(1)$ is divisible by $3$, which is a contradiction. Therefore, $Q/\ker(\theta)$ is abelian as claimed. This implies that $Q'\leq \ker(\theta)$, and so $Q'\leq M$. Thus $\Q=Q/M$ is abelian, and this contradicts the fact that  $\Q'$ is a group of order $q^{2}$, see \cite[Theorem]{a:Ward-66}. Therefore, $\In_{G'}(\theta)=G'$.
	
	We now show that $M=M'$. Since every linear character $\theta$ of $M$ is $G'$-invariant, we apply Lemma~\ref{lem:schur} and conclude that $|M/M'|$ divides the order of Schur Multiplier $\Msf(H_{0})$ of $H_{0}$ which is trivial by Lemma~\ref{lem:out}, and so $M=M'$. Suppose that $M$ is non-abelian, and let $N\leq M$ be a normal subgroup of $G'$ such that $M/N$ is a chief factor of $G'$. Then $M/N\cong S^{k}$, for some non-abelian simple group $S$. It follows from Lemma~\ref{lem:exten} that $S$ possesses a non-trivial irreducible character $\theta$ such that $\chi:=\theta^{k}\in \Irr(M/N)$ extends to $G'/N$. By Lemma~\ref{lem:gal}(ii), we must have $\chi(1)\psi(1)\in \cd(G'/N)\subseteq \cd(G')$, for all $\psi \in \Irr(G'/M)$. Now we can choose $\psi\in G'/M$ such that $\psi(1)=q^{3}$, and since $\theta$ is non-trivial, $\chi(1)\psi(1)=\theta(1)^{k}\cdot q^{3}$ divides some degree of $G$, which is a contradiction. Therefore, $M$ is abelian, and since $M=M'$, we conclude that $M=1$. Consequently, $G'$ is isomorphic to $H_{0}$.
\end{proof}

We are now ready to prove Theorem~\ref{thm:main}. In what follows, suppose that  $G$ is a finite group with $\cd(G)=\cd(H)$, where $H$ is an almost simple group whose socle is $H_{0}=\, ^{2}\G_{2}(q)$ with $q=3^{f}\geq  27$ and $f$ odd. Suppose also that $|H:H_{0}|=d$. \medskip

\noindent {\rm \textbf{Proof of  Theorem~\ref{thm:main}}}
The case where $H=H_{0}$ has been treated in \cite{a:Wakefield-2011-2G2}. Here, we assume that $H\neq H_{0}$. By Proposition~\ref{prop:step-1}, we conclude that $G'=G''$, and so $G$ is nonsolvable. Moreover, Proposition~\ref{prop:step-3} implies that the chief factor $G'/M$ is isomorphic to $H_{0}$. Suppose now that $f$ is coprime to $3$. Then Proposition~\ref{prop:step-3} implies that $G'\cong H_{0}=\, ^{2}\G_{2}(q)$. It follows from Lemma~\ref{lem:as} that $T:=G'\Z(G)=G'\times \Z(G)$ and $G/\Z(G)$ is isomorphic to $H_{0}: \Cyc_{s}\leq \Aut(H_{0})$ with $s$ a divisor of $d=|H:H_{0}|$. By Lemma~\ref{lem:deg}(v), the group $H$ has an irreducible character $\chi$ of degree $(q^{3}+1)d$. Then by Lemma~\ref{lem:gal}(i), there exists $\theta\in\Irr(T)$ such that $\chi(1)=k\theta(1)$ for some divisor $k$ of $|G:T|=s$. Note that $\cd(T)=\cd(G')=\cd(^{2}\G_{2}(q))$. Then by inspecting the degrees of $ ^{2}\G_{2}(q)$, we conclude that $\theta(1)=q^{3}+1$, and so $k=d$. This in particular implies that $d\leq s$. If $s>d$, then $G/\Z(G)\cong H_{0}:\Cyc_{s}$ has an irreducible character of degree $(q^{3}+1)s$, so does $G$, which is a contradiction. Therefore, $s=d$, and hence $G/\Z(G)$ is isomorphic to $H=H_{0}:\Cyc_{d}$.\hfill \qed

\section*{Statements and Declarations}

The author confirms that this manuscript has not been published elsewhere. It is not also under consideration by another journal. The author also confirms that there are no known conflicts of interest associated with this publication and there has been no significant financial support for this work that could have influenced its outcome.
The author has no competing interests to declare that are relevant to the content of this article and he confirms that availability of data and material is not applicable.

\section*{Acknowledgments}
The author are grateful to the editor and the anonymous referees and  for careful
reading of the manuscript and for corrections and suggestions.


\end{document}